\documentclass{amsart}
\usepackage{amssymb}
\usepackage{latexsym, amsmath,amsthm, rotating, longtable}

\theoremstyle{plain}
\newtheorem{thm}{Theorem}[section]
\newtheorem{prop}[thm]{Proposition}

\theoremstyle{definition}

\newtheorem{question}{Question}
\newtheorem{rem}[thm]{Remark}
\newtheorem{example}{Example}

\def\QQ{{\mathbb{Q}}}
\def\CC{{\mathbb{C}}}
\def\ZZ{{\mathbb{Z}}}

\DeclareMathOperator{\im}{Im}
\DeclareMathOperator{\Hom}{{Hom}}
\DeclareMathOperator{\End}{{End}}

\def\A{\mathcal{A}}

\providecommand{\dim}{\mathop{\rm dim}\nolimits}
\providecommand{\triv}{\mathop{\rm triv}\nolimits}

\begin{document}

\title{Completely decomposable Jacobian varieties in new genera} 

\author[J. Paulhus]{Jennifer Paulhus}
\address{%
Department of Mathematics and Statistics\\
Grinnell College\\
Grinnell, IA 50112\\
United States}

\email{paulhus@math.grinnell.edu }%

\author[A. M. Rojas]{Anita M. Rojas}
\address{%
Departamento de Matem\'aticas\\
Facultad de Ciencias\\
Universidad de Chile\\
Las Palmeras 3425\\
7800024 Nunoa\\
Santiago, Chile}

\email{anirojas@uchile.cl }%
\thanks{Partially supported by Fondecyt Grant 1140507}


\begin{abstract}
We present a new technique to study Jacobian variety decompositions using
subgroups of the automorphism group of the curve and the corresponding  intermediate covers. In
particular, this new method  allows us to produce many new examples of
genera for which there is a curve with completely decomposable Jacobian.
These examples greatly extend the list  given by Ekedahl and Serre of genera containing such curves, and provide more evidence for   a positive answer to two questions they asked.  Additionally, we produce  new examples of families of curves, all of which have completely decomposable Jacobian varieties.  These families relate to questions about special subvarieties in the moduli space of principally polarized abelian varieties.
\end{abstract}

\maketitle

\section{Introduction}

A principally polarized abelian variety   over $\CC$ is called
{\it completely decomposable} if it is isogenous to a product of elliptic curves.  Ekedahl and Serre \cite{ekserre} ask the following two questions. 

\begin{question} Is it true that, for all whole numbers $g$, there exists a curve of genus $g$ whose Jacobian is completely decomposable?
\end{question}

\begin{question} Are the genera of curves with completely decomposable Jacobians bounded?
\end{question}

They demonstrate various curves up to genus $1297$ with completely
decomposable Jacobian varieties.  However, there are numerous genera in that range for which they do not produce an example of a curve with this property.  

Since their paper, there has been much interest  in  curves with completely
decomposable Jacobian varieties, particularly the applications of such
curves to number theory. Dimension two has been widely studied; for
example, in \cite{earle}  a full classification of Riemann matrices of {\it
  strictly} completely decomposable Jacobian varieties of dimension $2$ is
given (these are  Jacobians which are {\it isomorphic} to a product of
elliptic curves). In \cite{kaniE}, the case of completely decomposable
abelian surfaces is studied, and several other authors have also studied  these questions.  See  \cite{crr},  \cite{msv}, \cite{nakajima}, and \cite{yama}, among many others.

Additionally, in \cite[Question 6.6]{mo}
the authors ask about positive dimensional special subvarieties, $Z$, of the
closure  of the Jacobian locus  in the
moduli space of principally polarized abelian varieties 
such that the abelian variety corresponding with the geometric generic
point of $Z$ is isogenous to a product of elliptic curves.  In Section \ref{S:families}, we discuss examples of positive dimensional families of curves with completely decomposable Jacobians, and connections to this question.

Despite advancements in the field, the questions of \cite{ekserre} still
remain open. Since the
publication of Ekedahl and Serre's list of genera 20 years ago, there have been few new
examples of genera with a compact Riemann surface of that genus with a
completely decomposable Jacobian. In
\cite{yama}, the author
gives a list of integer numbers $N$ such that the Jacobian variety $J_0(N)$
of the modular curve $X_0(N)$ has elliptic curves as $\mathbb{Q}-$simple
factors. These examples include  three genera not previously noted in \cite{ekserre} for which there is a completely
decomposable Jacobian variety: these are genus $113, 161$, and $205$
(corresponding to $N=672,1152$, and $1200$, respectively). His techniques are
number theoretic and relate to \cite[Section 2]{ekserre}.

 In this paper, we use experimental tools to find many examples of completely decomposable
Jacobian varieties in new genera.  We do this by using a new approach involving known results on {\it intermediate coverings}, i.e.,
quotients by the action of subgroups of the full group acting on the variety.   \\

We summarize the main results of this work in the following theorem. The bold numbers indicate genera which are new in this paper.\ \\

\noindent {\bf Theorem.} 
For every $g \in \{$1--29, {\bf 30}, 31, {\bf 32}, 33, {\bf 34--36}, 37, {\bf 39}, 40, 41, {\bf 42}, 43, {\bf 44}, 45, {\bf 46}, 47, {\bf 48}, 49, 50, {\bf 51--52}, 53, {\bf 54}, 55, 57, {\bf 58}, 61, {\bf 62--64}, 65, {\bf 67}, {\bf 69}, {\bf 71--72}, 73, {\bf 79--81}, 82,
{\bf 85}, {\bf 89}, {\bf 91}, {\bf 93}, {\bf 95}, 97, {\bf 103}, {\bf 105--107}, 109, {\bf 118}, 121, {\bf 125}, 129, {\bf 142}, 145,  {\bf 154}, 161, 163,
{\bf 193}, {\bf 199}, {\bf 211}, {\bf 213}, 217, {\bf 244}, 257, 325, 433$\}$  there is a curve of genus $g$ with completely
decomposable Jacobian variety which may be found using a group acting on a curve. Moreover in some cases there is a family of
such curves of dimension greater than $0$.  \\

Even for the already known genera, our examples are all derived from group actions, whereas many of the examples found in \cite{ekserre} use the theory of modular curves.  Genus 3 to 10, except genus $8$ are in \cite{paulhus}. Genus $8$ may be
found with a curve of automorphism group of size $336$  where the Jacobian
is isogenous to 
$E^8$ for some elliptic curve $E$.   \\

The pervious theorem, and the approach we outline in Section \ref{S:intcoverings}, open up the possibility that Question 1 might have a positive answer, and that group actions might be the tool to solve it. As we will see, once there is a completely decomposable Jacobian variety of a certain genus, by considering subgroups, it is possible to produce new examples  in lower genus.  Moreover, this theorem also suggests that perhaps there is no bound for the genus of curves with completely decomposable Jacobian variety, it might be a matter of simply finding the right group action.  \\

We describe the techniques used to decompose Jacobians in sections
\ref{S:gpactions} and \ref{S:intcoverings}.  In Section \ref{S:results}
we give explicit examples in both new  and old genera. Our new examples may
be found in Theorem \ref{T:purple} and Theorem \ref{T:redgreen}. Those
genera exhibiting a family of curves of dimension greater than $0$ with completely decomposable Jacobians
are given in Theorem \ref{T:family}. The computations needed to find both the old and new
examples were made using Magma \cite{magma}.  Finally, we address computational limitations of our techniques in Section \ref{S:complications}. The many examples from the paper may be useful to researchers interested in open questions surrounding groups acting on curves with completely decomposable Jacobians.

\section*{Acknowledgments} 
The second author is very grateful to Grinnell College, where the final version of this paper was written, for its hospitality and the kindness of all its people.

\section{Techniques} \label{S:technique}

Consider a  compact Riemann surface $X$ (referred to from now on as a
\lq\lq curve\rq\rq) of genus $g$ with a finite group $G$ acting on that
curve.  We write the quotient curve $X/G$ as $X_G$ and the genus of the quotient as $g_0$. Let the cover $X \to X_G$ be ramified at  $r$ places, $q_1, \ldots, q_r$. The signature of the cover is an $(r+1)$-tuple $[g_0;s_1, s_2,\ldots, s_r]$ where the $s_i$ are the ramification indices of the covering at the branch points. We denote the Jacobian variety of $X$ by $JX$.

\subsection{The group algebra decomposition.}\label{S:gpactions}
For many  examples, we use the group action of the automorphism group $G$ of $X$ to decompose $JX$.  We briefly describe the technique here for a general abelian variety $A$.  More details may be found in the original article \cite{lr} or \cite[Chp. 13]{lb}.

Let $A$ be an abelian variety of dimension $g$ with a faithful action of a finite group $G$. There is an induced homomorphism of $\QQ$-algebras
$$
 \rho: \QQ[G] \to \End_{\QQ}(A).
$$ Any element $\alpha \in \QQ[G]$ defines an abelian subvariety
$$
A^{\alpha} := \im (m\rho(\alpha)) \subset A
$$
where $m$ is some positive integer such that $m\rho(\alpha) \in \End(A)$. This definition does not depend on the chosen integer $m$.

Begin with the decomposition of $\QQ[G]$ as a product of simple
$\QQ$-algebras $Q_i$
$$
\QQ[G] = Q_1 \times \cdots \times Q_r.
$$  The factors $Q_i$ correspond canonically to the rational irreducible
representations $W_i$ of the group $G$, because each one is generated by a
unit element $e_i \in Q_i$ which may be considered as a central idempotent
of $\QQ[G]$.

The corresponding decomposition of $1 \in \QQ[G]$,
$$
1 = e_1 + \cdots + e_r
$$
induces an isogeny, via $\rho$ above, 
\begin{equation} \label{eq2.1}
A^{e_1} \times \cdots \times A^{e_r} \to A
\end{equation}
which is given by addition. Note that the components $A^{e_i}$ are
$G$-stable complex subtori of $A$ with $\Hom_G(A^{e_i},A^{e_j}) =0$ for $i
\neq j$. The decomposition \eqref{eq2.1} is called the {\it isotypical decomposition} of the complex $G$-abelian variety $A$.

The isotypical components $A^{e_i}$ can be decomposed further, using the
decomposition of $Q_i$ into a product of minimal left ideals. If $W_i$ is
the irreducible rational representation of $G$ corresponding to $e_i$ for
every $i = 1, \dots, r$, and  $\chi_i$ is the character
of one of the irreducible $\CC$-representations associated to $W_i$, then set

$$
n_i = \frac{\deg \chi_i}{m_i}  
$$ where $m_i$ denotes the Schur index of $\chi_i$. There is a set of primitive idempotents $\{\pi_{i1}, \cdots,\pi_{in_i}\}$
in $Q_i \subset \QQ[G]$ such that
$$
e_i = \pi_{i1} + \cdots + \pi_{in_i}.
$$ Moreover, the abelian subvarieties $A^{\pi_{ij}}$ are mutually isogenous
for fixed $i$ and $j= 1, \dots, n_i$. Call any one of these isogenous
factors $B_i$.  Then (see \cite{cr})
$$
B_i^{n_i} \to A^{e_i}
$$
is an isogeny for every $i = 1, \dots, r$. Replacing the factors in \eqref{eq2.1} we get an isogeny called the {\it  group algebra decomposition} of the $G$-abelian variety $A$

\begin{equation} \label{eq2.4}
B_{1}^{n_1} \times \cdots \times B_{r}^{n_r} \to A.
\end{equation}

Note that, whereas \eqref{eq2.1} is uniquely determined, \eqref{eq2.4} is not.
It depends on the choice of the $\pi_{ij}$ as well as the choice of the
$B_i$.  However, the dimension of the factors will remain fixed.

\begin{rem}\label{R:dims} While the factors in \eqref{eq2.4} are not necessarily easy to
determine, we may compute their dimension in the case of a Jacobian variety $JX$ with the action of a group
$G$ induced by the action on the corresponding Riemann surface $X$ (see \cite{paulhus} for
details). Define
$V$ to be the representation of $G$ on $H_1(X,\ZZ) \otimes_{\ZZ} \QQ$.  As
mentioned at the beginning of this section, here we assume the quotient
$X_G$ has genus $g_0$ and the cover $\pi:X\to
X_G$ has $r$ branch points $\{ q_1 \ldots, q_r\}$ where each $q_i$ has
corresponding monodromy $g_i$.  The tuple $(g_1, \ldots, g_r)$ is called the generating vector for the action \cite{brou}.

Then the
character $\chi_V$ associated to $V$ is  
\begin{equation}\label{eq:chi}  \chi_V=2\chi_{\triv}+2\left(g_{0}-1\right)\rho_{\langle
  1_{G}\rangle}+\sum\limits_{i=1}^{r}\left(\rho_{\langle
    1_{G}\rangle}-\rho_{\langle g_{i}\rangle}\right)\end{equation}
\cite[Equation 2.14]{brou}, where $\chi_{\triv}$ is the trivial character on $G$, $\rho_H$ is the induced character on $G$ of the trivial character of
the  subgroup $H$ (when $H=\langle g_i \rangle$, this subgroup is the stabilizer, or isotropy
group, of a point
in the fiber of the branch point $q_i$), and $\rho_{\langle 1_{G} \rangle}$
is the character of the regular representation. Then

\begin{equation}\label{eq:dims}
\dim B_i=\frac{1}{2}\dim_\QQ \pi_{i,j} V=\frac{1}{2} \langle \psi_{i},\chi_{V} \rangle
\end{equation} 
where here $\psi_{i}$  is the character of the $\QQ$-irreducible
representation of $G$ corresponding to $B_i$, and $\chi_V$ is the character defined in \eqref{eq:chi}.  
\end{rem}

One way we find completely decomposable Jacobian varieties is to search for
curves so that the decomposition in \eqref{eq2.4}  gives factors $B_i$ of
dimension only $0$ or $1$, computed via \eqref{eq:dims}.

\subsection{Intermediate Covering Decomposition}\label{S:intcoverings}

While the technique in the previous section gives us a number of examples
of completely decomposable Jacobians in new genera (see Theorem \ref{T:purple}),
we can extended the technique by studying decompositions of intermediate
coverings of a higher genus curve with a known decomposition of its
corresponding Jacobian variety. This idea expands the range of genera with
completely decomposable Jacobians which can be found using group
actions. We find many more new genera, as listed in Theorem \ref{T:redgreen}.  

To describe the technique, we begin with the  following proposition.  

\begin{prop} \cite[Proposition 5.2]{cr}  Given a Galois cover $X \to X_G$,
  consider the group algebra decomposition \eqref{eq2.4} 
$$JX \sim B_{1}^{\frac{\deg \chi_1}{m_1}} \times \cdots \times B_{r}^{\frac{\deg \chi_r}{m_r}}.$$
If $H$ is a subgroup of $G$ and $\pi: X \to X_H$ is the corresponding
quotient map, then the group algebra decomposition of $JX_H$ is given as
\begin{equation}\label{eqdimJH} JX _H\sim B_{1}^{\frac{\dim V_1^H}{m_1}} \times \cdots \times
B_{r}^{\frac{\dim V_r^H}{m_r}}\end{equation} where $V_j$ is a complex irreducible
representation associated to $B_j$, and  $V_j^H$ is the subspace of $V_j$ fixed by $H$. \end{prop}

By  Frobenius Reciprocity, we know that
\[ \dim {V_j}^H=\langle V_j, \rho_H\rangle \]
where $\langle V_j, \rho_H\rangle$ is the inner product of the characters
of these representations. Suppose $X$ is a curve with a known Jacobian
decomposition as in \eqref{eq2.4}, not necessarily completely
decomposable.  Then apply the previous proposition to get a
decomposition of $JX_H$ as in \eqref{eqdimJH} where
$JX_H$ will be completely decomposable precisely when $\langle V_j, \rho_H
\rangle=0$ for all $j$ such that $\dim B_j>1$ for the $B_i$ in the decomposition of $JX$. We have thus proven:

\begin{prop}\label{P:conditionintermediatecover} Given the conditions in the previous proposition, assume that 
$\langle V_j, \rho_H\rangle=0 $ for all $j$ such that $\dim B_j>1$, and $\langle V_j, \rho_H\rangle\neq 0 $ for at least one index $j$ with $\dim B_j=1$. Then the Jacobian variety of the curve $X_H$ is completely decomposable.
\end{prop}

Notice that even though a Jacobian variety $JX$ may not be  completely
decomposable, a Jacobian $JX_H$ of some intermediate cover $X_H=X/H$ could
decompose completely.  This gives us a much richer set of curves to search
through to
find completely decomposable Jacobian varieties. There are numerous
examples of  curves in high genus  whose Jacobians
decompose into many elliptic curves, but may not be themselves completely
decomposable. By applying Proposition \ref{P:conditionintermediatecover}, quotients of these curves may then be completely
decomposable. \\

Let us demonstrate with a couple of examples. More details and several other examples may be found in Section \ref{S:redgreenex}. First, a note on our notation for the rest of the paper. In most instances,  we will write a specific group as an ordered pair, where the
first number is the order of the group and the second number is its number
in the Magma or GAP database of groups of small order.  Also, for consistency in labeling subgroups and conjugacy classes, we convert all groups to permutation groups first.  For the Jacobian decompositions, when we write $E^{n} \times  E^m$ we are assuming that $E^n$ corresponds to one factor $B_i^{n_i}$ from \eqref{eq2.4} and $E^m$ corresponds to a different factor $B_j^{n_j}$ in \eqref{eq2.4}.  It is possible that in some cases these elliptic curves are isogenous.   \\

\begin{example} A complete search of genus $12$ curves as listed in \cite{breuer} using
techniques from Section \ref{S:gpactions} gives no example of a genus $12$
curve with a completely decomposable Jacobian. However, we may find one
as the quotient of a higher genus curve which has a completely decomposable
Jacobian. There is a curve $X$ of genus $29$ with the action of $G=(672,1254)$. From Section \ref{S:gpactions}, the Jacobian of $X$ decomposes completely as
\[JX \sim E^6 \times E^7\times E^8\times E^8,\]
each factor corresponding to a different complex (and rational) irreducible
representation. The group $G$ has several non-normal subgroup $H$ of order $2$, and one is such that
the dimensions of the fixed spaces for the corresponding representations
from the decomposition are all $3$.  Therefore the Jacobian of the
intermediate curve $X_H$ (a genus $12$ curve) decomposes as the same four elliptic curves each one to the power of $3$. That is,
\[J(X_H) \sim E^3 \times E^3\times E^3\times E^3.\]
Note that Ekedahl and Serre also find a genus $12$ example as a quotient of the modular curve $X_0(198)$ of genus $29$ by an involution. However, the group in  our
example is too  large to be the  automorphism group of this modular curve. \end{example}

\begin{example} Using this technique on one of our new examples from Section \ref{S:gpactions},  we
  can generate another example. Consider $G=(720,767)$ acting on a
  curve $X$ of genus $61$ with signature $[0;2,6,6]$. It has a subgroup $H$ of order $2$ such that
  $X_H$ has genus $30$ and a completely decomposable Jacobian. Note that a genus $30$  example is not one found by Ekedahl and Serre. \end{example}

\begin{example} Finally consider an example of a higher genus curve which is not completely decomposable, but an intermediate cover produces a lower genus curve which is completely decomposable. There is a genus $101$ curve with automorphism group $G=(800,980)$ whose Jacobian decomposes as
$$JX \sim E \times A_2 \times E^2 \times \underbrace{E^8 \times \cdots \times E^8}_{12}$$
where $A_2$ is an abelian variety of dimension $2$.  This group has three subgroups which produce quotients of genus $51$.  One of those three subgroups produces a decomposition as in \eqref{eq2.4} where the factor above of dimension $2$ has dimension $0$, and thus we get the following complete decomposition:
$$JX_H \sim E \times E^2 \times \underbrace{E^4 \times \cdots \times E^4}_{12}.$$ \end{example}

\section{Results}\label{S:results}

In this section we apply the techniques from Sections \ref{S:gpactions} and
\ref{S:intcoverings} to find completely decomposable Jacobian varieties
(including  all the genera found by  Ekedahl and Serre
\cite{ekserre} and Yamauchi \cite{yama} except for $g=113$, $205$, $649$, and $1297$).

Our primary task is to find examples where the dimensions in \eqref{eq2.4}
or \eqref{eqdimJH} above are all $0$
or $1$.  To construct our examples, we must know the automorphism group and signature of curves in high genus.  We use three  data sources for this information. Breuer
\cite{breuer} provides complete lists of
automorphism groups and signatures for curves of a given genus up to genus $48$. We use his data up through genus $20$. For genus 21--101, we
use data  computed by Conder \cite{conder},  giving  all 
automorphism group of size  greater than $4(g-1)$ for a given genus $g$ (this size condition guarantees, in particular,
 that $g_0$ is $0$).

 Finally, for genus greater than $101$, we use the ideas
described in \cite{condergpactions} to find possible automorphism
groups corresponding to  a few targeted  signatures (particularly those signatures which gave us
lower genus examples as in Theorem \ref{T:purple}).  Given a signature
$[0;s_1, \ldots, s_r]$, we use the Magma
command \verb+LowIndexNormalSubgroup(K,n)+  to find all possible
low index normal subgroups of the group $$K=\langle x_1, \ldots, x_r |
x_1^{s_1}= \cdots = x_r^{s_r} = x_1\cdots x_r=1 \rangle.$$ These normal subgroups  give us
possible automorphism groups for that signature up to a particular genus
which depends on the choice of $n$ (see \cite[page 260]{fk} ). We will see that these large genus curves give us many new examples. \\

Notice that the computation of $\chi_V$ in \eqref{eq:chi} requires
knowledge of a generating vector of the action. Modifications to \cite{breuer}  give us a way to
compute generating vectors if the automorphism group and signature are already known.  See
\cite{branching} for details.   

For each of these three data sets and a fixed group $G$ and signature, we first compute the Jacobian
decomposition as in \eqref{eq2.4} and, if this is completely decomposable,
we record it.  Next we compute all orders of subgroups of $G$ and if any of
those orders produce a quotient of still unknown genus, we apply the technique
of Section \ref{S:intcoverings}  to see if this subgroup produces a
completely decomposable intermediate cover.  Note that, from \eqref{eqdimJH},  if we take a completely
decomposable Jacobian of higher genus, the Jacobian
variety corresponding to any intermediate quotient by any subgroup will automatically be
completely decomposable.

In our computations, as we increased the genus, we removed from consideration all lower genera we had already
found an example for.  So our examples for Section \ref{S:intcoverings} are
just a sample of such curves for a given genus and may not represent all
curves of that genus which have decomposable Jacobians realizable through
group actions.   We chose as our goal demonstrating the
usefulness of our technique, and
not performing an exhaustive search of all decomposable Jacobians for any
known genus.

We divide the results into three sections:  those found through the
technique in \S 2.1, those found through the technique in \S 2.2, and those
which give a  family of dimension greater than $0$ of completely decomposable Jacobians of
a given genus.

\subsection{Group algebra decomposition examples}

The new genera, those not included in Ekedahl and Serre's paper,  found
using the technique in Section \ref{S:gpactions} are given here.

\begin{thm}\label{T:purple} 
Let $g\in \{36, 46, 81, 85, 91, 193, 244 \}.$ There is a completely
decomposable Jacobian variety of dimension $g$. Moreover, each one
corresponds to the Jacobian variety of a curve of genus $g$ with the action
of a group $G$ as listed in Table \ref{Tb:purple}. The signature for the
action and the decomposition are also listed in the table.
\end{thm}
\begin{proof} 
The proof consists of following the program outlined in Section \ref{S:gpactions}. For this we need to find appropriate group actions on the missing genera.

In genus $36$ there are two curves, up to topological equivalence with automorphism group PGL$(2,7)$ and signature   $[0;2,6,8].$   
It is possible to classify  actions  topologically by using the action of the Braid group on a generating vector for the action. We do not describe these details here, but references are
\cite{brou}, \cite{harvey} and \cite{vbook}. A review of the principal results on this matter and a program in Sage \cite{sage} which computes the non-equivalent actions, can be found in \cite{camila}, \cite{bmr}. Both curves have a decomposition of the form
\[JX \sim E^6 \times E^7 \times E^7 \times E^8 \times  E^8. \]

The irreducible $\CC$-characters of this group are all irreducible
$\QQ$-characters, except for two of degree $6$. The decomposition then follows from \eqref{eq2.4} and  \eqref{eq:dims}. The factors in the Jacobian
decomposition come from an irreducible $\CC$-character of degree
$6$, and the  two irreducible $\QQ$-characters in each of degrees $7$ and
$8$. 

In genus $46$ there is one curve, up to topological equivalence, with automorphism group $(324,69)$ and signature $[0;2,6,18]$. This curve has a decomposition
\[JX \sim E \times E \times E^2 \times  E^6 \times E^6 \times E^6 \times E^6 \times E^6 \times E^6 \times E^6 .\]
 In this case, 
the factors in this decomposition come from two separate sets of two degree $1$ irreducible $\CC$-characters whose sums are irreducible $\QQ$-characters, one set of two degree $2$ irreducible $\CC$-characters whose sum is also an irreducible $\QQ$-character, and six of the irreducible $\CC$-characters of degree $6$ which are all irreducible $\QQ$-characters.

In genus $81$ a curve $X$ with automorphism group  $(1152,157853)$  and signature $[0;2,4,9]$  has Jacobian  decomposition
\[JX \sim E^9 \times E^9 \times E^9 \times  E^9 \times E^9 \times E^9 \times E^9 \times E^9 \times E^{9}.\] For genus $85$ there is a curve $X$ with automorphism group of size $2016$ given as \cite{conder}  {\tiny $$\langle x,y,z | x^2  z^{-1} \cdot y^{-1} \cdot x, y^4, z^6, y^{-1} \cdot z \cdot y \cdot
x \cdot z^2 \cdot y \cdot x \cdot y^{-1} \cdot z \cdot y^{-1} \cdot z^{-2}
\cdot x \cdot z,  y\cdot z^{-1} \cdot y \cdot z^{-1} \cdot y \cdot
z^{-1}\cdot x \cdot y^2 \cdot z^{-1} \cdot y \cdot z^{-1} \cdot y \cdot
z^{-2} \rangle$$} and with signature $[0;2,4,6]$ which has Jacobian decomposition 
\[E^6 \times E^{7} \times E^{8} \times  E^{8} \times E^{12} \times E^{14} \times E^{14} \times E^{16}.\]

In the genus $81$ case, the factors in this decomposition come from nine separate  degree $9$ irreducible $\CC$-characters  which are all irreducible $\QQ$-characters. In the genus $85$ case,  the factors in the decomposition come from irreducible $\CC$-characters one each of degree $6$, $7$, $12$, and $16$, and two each of degree $8$ and $14$.  All of these characters are irreducible $\QQ$-characters.

For genus $91$ there is a one dimensional family of curves with automorphism group $G=(432,686)$ and signature $[0;2,2,2,12]$. All curves in this family are completely decomposable. Using data from \cite{conder} for genus $91$, there is no larger automorphism group which has curves with completely decomposable Jacobians.  In particular, no curve in this family has a larger automorphism group. 

Finally, for genus $193$ there is an curve with automorphism group of size 5706 and signature  $[0;2,3, 10]$  while in  genus $244$,   the size of the group is $11,664$ and the signature is
$[0;2,3,8]$.  Both examples were found using the Magma command \verb+LowIndexNormalSubgroup+ to determine the automorphism groups.  Here is the presentation of the group for genus $193$:
\tiny
$$ \langle x,y,z | x^2, y^3, z^{10}, z^{-1}  y^{-1}  x, x  z^2  y  z^{-1}  x  z  y^{-1}  z^{-2}  x  z  y^{-1}  z^{-2}, y  z^{-1}  x  z^4  y  z^{-1}  x  y^{-1}  z^{-1}  x  y^{-1}  z^{-2}  x  z^4  y  z^{-1}  x  y^{-1}  z^{-1}  x  z,$$ $$z^2  y^{-1}  z^{-4}  x  z  y  z^{-1}  x  z^2  y  x  z  y  x  z^{-1}  x  z  y^{-1}  z^{-3}  x \rangle,$$\normalsize
and here is the presentation for the group of genus $244$:
\tiny
$$\langle x,y,z | x^2, y^3, z^8, z^{-1}  y^{-1}  x, z  y  x  z  y  x  z  y  x  y^{-1}  z^{-1}  x  y^{-1}  z^{-1}  x  y^{-1}  z^{-1}  x, z^2  y  x  z^2  y  x  z^2  y  x  z^2  y  x  y^{-1}  x  y^{-1}  z^{-1}  x  z  y^{-1}  z^{-1}  x \rangle .$$
\normalsize
\end{proof}

In Table \ref{Tb:purple}  we record one example of a curve with completely
decomposable Jacobian  for  each genus found with the technique from section \ref{S:gpactions}. For completeness, we include the genera found by Ekedal and Serre, or Yamauchi. For each genus, we chose an example with the largest automorphism group.  In the table we include the automorphism group as well as
the signature. When possible, we denote the groups as ordered pairs where the first term
is the order of the group, and the second term is the group identity number
from Magma's database. If the order of the group exceeds the allowable sizes for these databases, we have labeled the group as a number (sometimes with a subscript).  The number represents the order of the group.  If the subscript itself is a number, then the group presentation may be found in data of Conder  \cite{conder} where the subscript denotes which of groups of that order (and with the corresponding signature) in his data it is.  If the subscript is a letter (or if there is no subscript at all), the presentation of the group may be found at \cite{webdata}.

The final column of the table represents the decomposition as a list of numbers represent the $n_i$ from \eqref{eq2.4}. Again, we note that it is conceivable that distinct elliptic curve factors in \eqref{eq2.4}  may be isogenous. The  new examples from this
technique are denoted by a *.

All of our examples come from  group actions,
while some of Ekedahl and Serre examples (and the newer work of  Yamauchi \cite{yama}) use
modular curves. We checked that in the genera where examples were obtained with modular curves
in \cite{ekserre}, our corresponding example was not a  modular
curve. To determine this,  we compared the size of the automorphism group of
modular curves of the relevant level,  which can be  determined by
using  \cite[Theorem 0.1]{modular1} and \cite[Proposition 2]{modular2},
with the size of the automorphism groups of our examples.  Only in $g=73$ did
the sizes match, and in that case we explicitly computed the automorphism
group of $X_0(576)$   to determine that it is not the same as our example
in Table \ref{Tb:purple}.   Notice that our genus $26$ example is the well known example of the curve $X(11)$.

\tiny
\begin{center}
\renewcommand{\arraystretch}{1.4}
\noindent \begin{longtable}{c c  l l }
\caption{Curves with completely decomposable Jacobians in genus greater than $10$, using group algebra decomposition. The examples are those we found with the largest automorphism group for that genus. }\label{Tb:purple}  \\
 {Genus} & {Automorphism Group} & {Signature} & {Jacobian Decomposition}\\ \hline
\endfirsthead
\caption{(continued)} \\ 
{ Genus} & { Automorphism Group} & {Signature} & {Jacobian Decomposition}\\ 
  \hline
\endhead
11 & $(240, 189)$ & $[0;2,4,6]$ & $5,  6$ \\
13 & $(360,121)$ & $[0;2,3,10]$ & $5, 8$ \\ 
14 & $(1092,25)$ & $[0;2,3,7]$ & $14$ \\
15 & $(504, 156)$ & $[0; 2,3, 9]$ & $7, 8 $ \\
16 & $(120,34)$  & $[0;3,4,6]$ & $ 5, 5, 6$\\
17 & $(1344, 814)$ & $[0;2,3, 7]$ & $3, {14}$\\ 
19 & $(720,766)$ & $[0;2,4,5]$ & $9, {10} $\\
21 & $(480, 951)$ & $[0;2,4,6]$ & 5, 6, 10 \\
22 & $(504, 160)$ & $[0;2,3,12]$ & $1, 3, {18}$ \\
24 & $(168,42)$ & $[0;3,4,7]$ & $3, 6, 7, 8$ \\ 
25 &  $(576, 1997)$ & $[0;2, 3, 12 ]$ &  $1, 2, 4, 6, 12$ \\  
26 & $(660,13)$ & $[0;2,3,11]$ & $5, 10, 11 $ \\
28 & $(1296, 2889)$ & $[0;2,3,8]$ & $2, 8, {18}$ \\
29 & $(672,1254)$  & $[0;2,4,6]$ &  $ 6, 7, 8, 8$\\
31& $(720, 767)$  & $[0;2,4,6]$ &  $5, 6, 8, 12$\\
33 & $(1536, 408544637)$ & $[0;2,3,8]$ & $2, 3, 12, {16}$ \\
36*  & $(336,208)$ & $[0;2,6,8]$ & $6, 7, 7, 8, 8$ \\
37 & $(1728, 31096)$& $[0;2, 3, 8 ]$ & $2, 3, 8, 24$ \\ 
41 & $(960,5719)$ & $[0;2,4,6]$ & $5, 6, 8, {10}, {12}$ \\
43  & $(672,1254)$ & $[0;2,4,8]$ & $6, 7, 7, 7,  8, 8$ \\
46*  & $(324,69)$ & $[0;2,6,18]$ & $1, 1,  2,   \underbrace{6,  \ldots, 6}_{7}$ \\
49 &  $(1920,240996)$ & $[0; 2, 4, 5]$ & $4, {10}, {15}, {20}$ \\
50 &  $(588,37)$ & $[0; 2, 6, 6]$ & $1, 1, 6, 6, {12}, {12}, {12}$\\
55 &  $(1296,3490)$ & $[0; 2, 4, 6]$ &  $3, {12}, {12}, {12}, {16}$ \\
57 &  $(1344,11289)$ & $[0; 2, 4, 6]$ & $6, 7, 8, 8, {12}, {16}$ \\
61 &  $(1440,4605)$& $[0; 2, 4, 6]$ & $ 2,5, 6, 8, 8, {10}, {10}, {12}$ \\ 
65 & $3072_1$ & $[0;2,3,8]$  & $2,3, {12}, {24}, {24}$ \\
73 &  $(1728,46270)$& $[0; 2, 4, 6]$ & $2,3, 4, 4, 4, 8, 8, {12}, {12}, {16}$  \\ 
81* &  $( 1152,157853)$ & $[0;2,4,9]$ &  $\underbrace{9, \ldots, 9}_{9}$ \\ 
82&$3888_2$& $[0; 2, 3, 8]$ & $2,8, 8, {16}, {24}, {24}$  \\
85&$4032_1$& $[0; 2, 3, 8]$ & $8, {14}, {18},  21, {24}$ \\ 
91& $(432,686)$ & $[0; 2, 2, 2, 12]$ &  $1, 2,2,2,\underbrace{4, \ldots, 4}_{21}$ \\ 
97 &$3840_1$& $[0; 2, 4, 5]$ & $4, {10}, {15}, {20}, {24}, {24}$ \\ 
109 & $2592_A$& $[0; 2, 4, 6]$ & $2,3, {12}, {12}, {12}, {12}, {16}, {16}, {24}$ \\ 
121 &2880& $[0; 2, 4, 6]$ & $3, 5, 6, 8, {12}, {12}, {12}, {15}, {15}, {15}, {18}$   \\ 
129&10752 & $[0; 2, 3, 7]$ & $3, {14}, {14}, 42, 56$ \\ 
145 &6912& $[0; 2, 3, 8]$ & $2,3, 8, {12}, {24}, {24}, {24}, 48$ \\  
163& $2592_C$& $[0; 2, 4, 8]$ & $1, 2,\underbrace{8, \ldots, 8}_{8}, \underbrace{{16}, \ldots, {16}}_{6}$ \\  
193*&5760& $[0; 2, 3, 10]$ & $5, 8, {15}, {15}, {15}, {15}, {30}, {30}, {30}, {30}$ \\ 
244* & 11664  & $[0; 2,3,8]$ &  $2,8, 8, 16, 24, 24, 36, 36, 36, 54$ \\
257 & $12288_A$ & $[0; 2, 3,8]$ &  $2,3, 12,\underbrace{24, \ldots, 24}_{6}, 48, 48$ \\
325 & 15552 & $[0; 2, 3, 8]$ & $ 2,3, 8, 8, 16, \underbrace{24, \ldots, 24}_{6}, 48, 48, 48$ \\
433 & 5184 & $[0; 2, 6, 6]$ &  $ 1, 1, 2,2,3,  4, \underbrace{6, \ldots,   6}_{8},  \underbrace{12,\ldots, 12}_{31}$ \\ 
\end{longtable}
\end{center}
\normalsize

Many more examples were found than appear in the paper.  We provide tables of  all examples we found, not just those of the largest automorphism group order, at \cite{webdata}.  For genus up to $20$, this is a complete list using this technique for all curves with $g_0=0$. For
genus  21 --101, this is a complete list for all curves with automorphism
group larger than $4(g-1)$. For genus beyond $101$ we only give the
curves  found by strategic searching, and there may be other examples for
a given genus.

  \subsection{Intermediate cover examples}\label{S:redgreenex} Using the technique from Section \ref{S:intcoverings}, we obtain the
following  new examples.  Notice that we found many more new genera with this new technique.  

\begin{thm}\label{T:redgreen}
Let $g\in \{$30, 32, 34, 35, 39, 42, 44, 48, 51, 52, 54, 58, 62--64, 67, 69, 71,
72, 79, 80, 89, 93, 95, 103, 105--107, 118, 125, 142, 154, 199,  211,  213$\}$. There is a completely decomposable Jacobian variety of dimension
$g$. Moreover, each one corresponds to the Jacobian variety of a curve
obtained as a quotient by $H\leq G$ of a curve of higher genus  with the action
of a group $G$.
\end{thm}

\begin{proof}
We give an outline of the proof for one case, the rest follow
similarly. Also recall that in Section \ref{S:intcoverings} we gave examples of several other cases.   

 Consider the group $G=(1152,5806)$ acting on a curve $X$ of genus $73$. Then 
$JX$ decomposes into $10$ factors (each one a power of an elliptic curve),
\[JX \sim E \times E^2 \times E^2 \times  E^4 \times E^8 \times E^8 \times E^8 \times E^8 \times E^{16} \times E^{16}. \]
Using the technique described in Section  \ref{S:intcoverings}, there is a non-normal subgroup $H$ of $G$ of order $2$ such that $X_H$ has genus $35$ and has a completely decomposable Jacobian. The decomposition of the Jacobian variety of the genus $35$ curve is as follows, where $E_i$ corresponds to the $i$th term in the decomposition of $JX$ above
\[J(X_H)\sim E_2 \times E_4^2 \times E_5^4 \times E_6^4 \times E_7^4 \times E_8^4 \times E_9^8 \times E_{10}^8.\]\end{proof}

In Table \ref{Tb:serregreenred} we give one example for each genus where we found an example only through intermediate covers.  We use the same convention for labeling groups as in Table \ref{Tb:purple}.  Again, complete lists of data we found are at \cite{webdata}. In this table we also include the genus of the
  intermediate cover, the genus and automorphism group and signature for
  the larger curve, the subgroup size and number for the corresponding subgroup $H$ (labeled as Magma does, and recall our convention of converting all groups to permutation groups), and the decomposition of the quotient
  curve.   For ease of notation, we grouped all factors from \eqref{eq2.4} of the same dimension together, although they may not be in that order, nor correspond to the order of the decomposition of the high genus curve. For instance, if the decomposition were given as $E^2 \times E^4 \times E^2$, we would denote this as $2, 2, 4.$ 

There are some genera (up to $500$) on Ekedahl and
  Serre's list for which the  technique in Section \ref{S:gpactions}
  cannot identify a curve with completely decomposable Jacobian, and which
  do not appear in Table \ref{Tb:purple}.  The set of such genera is $\{12$,
  $18$, $20$, $23$, $27$, $40$, $45$, $47$, $53$, $217\}$.   All these examples may be
  generated using our second technique of intermediate covers from
  Proposition \ref{P:conditionintermediatecover}. We also collect this  data in Table \ref{Tb:serregreenred}. Again, our new examples are denoted with a $*$. 

\tiny
\begin{center}
\noindent \begin{longtable}{c | c c l | l l } 
\caption{Examples of curves with completely decomposable Jacobians in genus greater than $10$ using intermediate coverings.} \label{Tb:serregreenred} \\
& &  Automorphism &  & Subgroups & Jacobian \\
$g$ &  Large $g$  &Group &  Signature &  No., Order & Decomposition \\  \hline
\endfirsthead
\caption{(continued)} \\ 
& &  Automorphism &  & Subgroups & Jacobian\\
$g$ &  Large $g$  &Group &  Signature &  No., Order & Decomposition \\  \hline
\endhead
12&49& $(288,627)$ & $[0;2,2,2,6]$ & 28, 4& $\underbrace{1, \ldots, 1}_{6}, 2,2,2$\\ 
18&73& $(1152,5806)$& $[0;2,4,8]$ &  35, 4 & 1, 1, 2, 2, 2, 2,  4, 4 \\
20&82& $3888_2$& $[0;2,3,8]$ &13, 4& 2, 2, 4, 6, 6 \\  
23&49& $(256,3066)$ & $[0;2,2,2,8]$ &9, 2& $1, 1, 1, \underbrace{2,
  \ldots, 2}_{10}$  \\ 
27 & 55 & $(432,537)$ & $[0;2,2,2,4]$ & 6, 2 & $1, 2, \underbrace{3, \ldots, 3}_8$ \\  
30*&61& $(720,767)$& $[0;2,6,6]$ &5, 2&  2, 2, 3, 3, 4, 5, 5, 6 \\ 
32*&97& $2304_6$ & $[0;2,3,12]$ &10, 3 & 2, 2, 4, 4, 4, 8, 8  \\ 
34*&73& $(432,682)$& $[0;2,2,2,6]$ &5, 2& $1, 1, \underbrace{2,\ldots, 2}_{16}$ \\ 
35*&73& $(1152,5806)$& $[0;2,4,8]$ &10, 2 & $1, 2, 4, 4, 4, 4, 8, 8$ \\ 
39*&81& $(1152,157853)$& $[0;2,4,9]$ &9, 2 & $ \underbrace{4, \ldots, 4}_6, 5, 5, 5$ \\ 
42*&129& $3072_F$& $[0;2,3,12]$ &11, 3 & $2, \underbrace{4, \ldots, 4}_{6}, 8, 8 $ \\ 
44*&91& $(432,686)$ & $[0;2,2,2,12]$ &7, 2&  $1,  1,  \underbrace{2,\cdots, 2}_{ 21 } $ \\ 
45&91& $(432,686)$ & $[0;2,2,2,12]$ &8, 2 & $1,  1, 1,  \underbrace{2,\cdots, 2}_{ 21 } $ \\ 
47&97& $3840_1$ & $[0;2,4,5]$ &5, 2&$2,4, 7, {10}, {12}, {12}$ \\ 
48*&145& $(1728, 13293)$ & $[0;2,6,6]$ &12, 3& $1,  1, \underbrace{2, \ldots, 2}_{7},\underbrace{4, \ldots, 4}_{8}$ \\ 
51*& 101 & $2400_1$ & $[0;3,3,4]$ &3, 2 & $3, {12}, {12}, {12}, {12}$ \\ 
52*&109& $2592_A$ & $[0;2,4,6]$ &7, 2& $1,  1, 5, 5, 6, 6, 8, 8, {12}$ \\ 
53&109& $(1296,2945)$ & $[0;2,6,6]$ &8, 2  & $1, 1, 1, 2, \underbrace{3, \ldots, 3}_{6}, 6, 6, 6, 6, 6$ \\ 
54*&109& $(1296, 3498)$ & $[0;2,4,12]$ &6, 2& $2, 2, 2, \underbrace{4, \ldots, 4}_{8}, 8, 8 $ \\ 
58* & 244 & 11664 & $[0;2,3,8]$ & 14, 4 & 2, 2, 4, 6, 6, 8, 8, 8, 14 \\ 
62* & 257 & $12288_B$ & $[0;2,3,8]$ & 35, 4 & 2, 4, 4, 6, 6, 8, 8, 12,12 \\ 
63* & 193 & 5760 & $[0;2,3,10]$ &9, 3& $1, 2, 5, 5, 5, 5, {10}, {10}, {10}, {10}$  \\ 
64* & 325 & 3888 & $[0;2,6,6]$ & 28, 4 & $1, \underbrace{3, \ldots, 3}_{ 21}$ \\ 
67*&145& $(1728, 32233)$ & $[0;2,6,6]$ &6, 2 & $1,  1, 2, 3, 3, 3, \underbrace{6,\ldots, 6}_{9}  $ \\  
69*&145& $(1728, 32233)$ & $[0;2,6,6]$ & 9, 2 & $1, 1, 1, 2, 2, 2, 3, 3, \underbrace{6, \ldots, 6}_{9}$ \\ 
71*&145& $(1728, 13293)$ & $[0;2,6,6]$ &8, 2 & $1,  1, \underbrace{3, \ldots, 3}_7,  \underbrace{6, \ldots, 6}_8 $ \\ 
72* & 325&  15552  & $[0;2,3,8]$ & 22, 4 & $1, 1, 4,\underbrace{5, \ldots, 5}_{6}, 12,12,12$  \\ 
79*&163& $2592_D$& $[0;2,4,8]$ &6, 2& $1, 3, 3, 4, 4, \underbrace{8, \ldots, 8}_{8}$ \\ 
80*&163& $2592_C$& $[0;2,4,8]$ &5, 2& $\underbrace{4, \ldots, 4}_{8}, \underbrace{8, \ldots, 8}_6$  \\ 
89*&193& 5760 & $[0;2,3,10]$ & 3, 2& $4, 5, 5, 5,  {10},  {15}, {15}, {15}, {15}$ \\ 
93*&193& 2304 & $[0;2,2,2,3]$ & 11, 2& $1, 1, 1, 1, 2, 3, 3, 3, \underbrace{4, \ldots, 4}_{12}, 6, 8, 8, 8$ \\ 
95*&193 & 2304 & $[0;2,2,2,3]$ &12, 2 & $1, 1, 1, 1, 1, 3, 3, 3, 3,  \underbrace{4,\ldots, 4}_{12},  6, 8, 8, 8$ \\ 
103* & 433 & 5184 & $[0;2,6,6]$ & 32, 4 & $1, 1, 1, 1,  \underbrace{2,\ldots,
  2}_{27}, 3, 3, 3, 3, 3, 6, 6, 6 ,6, 6$   \\ 
105* & 433 & 5184 & $[0;2,6,6]$  & 49, 4 & $\underbrace{1, \ldots, 1}_{8},  \underbrace{2, \ldots, 2}_{18},  3, \underbrace{4, \ldots, 4}_{13}, 6 $ \\
 106* & 325  & 15552 & $[0;2,3,8]$ & 14, 3 & $1, 2, 2, 5, \underbrace{8, \ldots,  8}_6,  16, 16, 16$  \\
107* & 433 & 5184 & $[0;2,6,6]$ & 39, 4 &  $\underbrace{1, \ldots, 1}_{10},  \underbrace{2, \ldots, 2}_{18},  3, \underbrace{4, \ldots, 4}_{13}, 6$  \\
118* & 244 & 11664 & $[0;2,3,8]$ & 3, 2 & 1, 3, 3, 8, 11, 11, 18, 18, 18, 27 \\ \
125* & 257 & $12288_B$ & $[0;2,3,8]$ & 6, 2  & 1, 4, 8, 8, 12, 12, 16, 16, 24, 24  \\ 
142* & 433 & 5184 & $[0;2,6,6]$ & 17, 3  & $1, 1, \underbrace{2,\ldots, 2}_8, \underbrace{4, \ldots, 4}_{31}$  \\ 
154* & 325&  15552 & $[0;2,3,8]$ & 4, 2 & $1, 1, 3, 3, 8, \underbrace{11, \ldots, 11}_6, 24, 24, 24$ \\ 
161 & 325 & 15552 & $[0;2,3,8]$ & 5, 2 & $1, 4, 4, 8, \underbrace{{12}, \ldots, 12}_6, 24, 24, 24$\\ 
199* &  433 & 5184 & $[0;2,6,6]$ & 5, 2 & $1, 1, 2, 3, 3, 3, \underbrace{6, \ldots, 6}_{31}$   \\ 
 211* & 433 & 5184 & $[0;2,6,6]$ & 7, 2 &  $1,  1,  2, \underbrace{3,  \ldots, 3}_7, \underbrace{6, \ldots, 6}_{31}$ \\ 
213* & 433 & 5184 & $[0;2,6,6]$ & 9, 2 &  $1,  1,  1,  2,  2, 2,  \underbrace{3, \ldots,  3}_6,  \underbrace{6, \ldots,  6}_{31}$ \\ 
 217* & 433 & 5184 & $[0;2,6,6]$ &  10, 2 & $1, 1, 1, 2, \underbrace{3,
   \ldots, 3}_6, 4, 4, \underbrace{6, \ldots, 6}_{31}$ \\ 
\end{longtable}
\end{center}

\normalsize

\subsection{Examples of Families}\label{S:families}

Recall from the proof of Theorem \ref{T:purple} that the only completely decomposable Jacobian varieties
of dimension $91$ discovered using the group algebra technique  were a one dimensional
family of curves (so using the group algebra technique only, there is no curve with an automorphism group 
corresponding to a dimension $0$ family in genus $91$ having a completely decomposable
Jacobian).  There are several
known examples of families of completely decomposable Jacobians in low genus  (see \cite{paola}, \cite[Section 4]{lro}, \cite{paulhusthesis}), and, as we mentioned in the introduction,   in \cite{mo}  the authors asked for examples of special subvarieties such that the generic point is completely decomposable. Therefore our techniques provide a way of finding examples of families where one can look for examples to answer their question.  

 Here we highlight the genera where we find a one
dimensional (or higher) family of completely decomposable Jacobians of that genus. We elaborate on the question in \cite{mo} after the theorem.

\begin{thm}\label{T:family}
Let $g\in \{$11--19, 21--29,  31, 33--35, 37, 40, 41, 43--47, 49, 52, 53, 55, 57, 61, 65, 57, 69, 73, 82, 91,
93, 95, 97, 109, 129, 145, 193$\}.$ Then there is a  dimension one (or larger) family
of completely decomposable Jacobians of curves of genus $g$ which can be found using the techniques from Sections \ref{S:gpactions}  and \ref{S:intcoverings}. \end{thm}

\begin{proof} Again, we only demonstrate with a couple of examples.  The
  rest follow in the same way via data listed in  Table \ref{Tb:familypurple} for those genera found
  through the technique in Section \ref{S:gpactions},  and  Table \ref{Tb:familyredgreen} for those found through
  the technique in Section \ref{S:intcoverings}.  In these tables, for each
  genus we only give an example of  the largest automorphism
   group we found which leads to a completely decomposable Jacobian (and
   the  highest dimensional family, if that is not the same). Again, for
   genus greater than $20$, we only searched groups of order greater than
   $4(g-1)$, so there may be other examples of higher dimensional families
   with completely decomposable Jacobians.   All other examples we found
   appear in the data at  \cite{webdata}. For completeness, we have
   added all examples for genus 3 through 10 curves in the Appendix, only including  those corresponding to the action of the full automorphism group.  This data includes many previously know examples.   
   
There is a family of curves of genus $73$ with the action of the group $(432,682)$
with signature $[0;2,2,2,6]$, see the data at \cite{webdata}.  Since this curve is completely decomposable, all quotients by   corresponding subgroups $H$ will also be completely decomposable.  In particular, this group has a subgroup of order $2$ which gives a new example for genus $34$.  

There are several different group actions on curves of genus $49$ giving  one dimensional families of completely decomposable Jacobians. For instance,  the group $(256,3066)$ acts with signature $[0;2,2,2,8]$ and has a subgroup of order $2$ which forms a quotient of genus $23$,  and the group $(288,627)$ acting with signature $[0;2,2,2,6]$ has a subgroup of order $4$ which forms a quotient of genus $12$.  \end{proof}

Using notation from earlier in the paper, let  $G$ be a finite group acting on genus $g$ with signature $m=[0;s_1,\dots,s_r]$, and generating vector $\theta=(g_1,\dots, g_r)$. For a fixed pair $(m,\theta)$, by moving the branch points of the covering in $\mathbb{P}^1$ one obtains an $(r-\!3)$-dimensional family of such coverings, and a corresponding family of Jacobians $\mathcal{J}(G,m,\theta)$ of the same dimension. For references see  \cite{paola} or \cite{vbook}. The symplectic group $Sp(2g,\mathbb{Z})$ acts on the Siegel upper half space, $\mathbb{H}_g$, and $\A_g=Sp(2g,\mathbb{Z})\setminus \mathbb{H}_g$ is a complex analytic space which parametrizes principally polarized abelian varieties  of dimension $g$. It corresponds to the analytic point of view of the moduli space of principally polarized abelian varieties over $\mathbb{C}$ of dimension $g$.

Denote by $Z(G,m,\theta)$ the closure of $\mathcal{J}(G,m,\theta)$ in $\A_g$. The action of $G$ on $X$, hence on its Jacobian $JX$, induces a symplectic representation $\rho:G\to Sp(2g,\mathbb{Z})$ of $G$. Let $\mathbb{H}_g^G$ be the set of fixed points of $G$ in $\mathbb{H}_g$ (see \cite{brr} for details). In \cite[Thms. 1.4, 3.9, Lemma 3.8]{paola} there is a nice characterization of when $Z(G,m,\theta)$ is a {\it special subvariety}. Their criterion is as follows, if the dimension of $\mathbb{H}_g^G$ equals the dimension of $\mathcal{J}(G,m,\theta)$, which is $r-3$, then $Z(G,m,\theta)$ is a special subvariety of $\A_g$ that it is contained in the closure $\mathcal{T}_g$ of the Torelli (or Jacobian) locus, and which intersects non-trivially the Torelli locus $\mathcal{T}_g^0$.

Given a pair $(m,\theta)$ for a fixed $G$, using \cite{brr} one can find the dimension of $\mathbb{H}_g^G$, although it is computationally expensive. Additionally, code is provided in \cite{paola} which can compute the dimension for low genus examples as well.   \\

Our Table 3 contains examples of families found using group actions, so we
can apply the criterion of \cite{paola} to determine if they correspond to
special subvarieties. We remark that this criterion is a sufficient one,
therefore these families could correspond to special subvarieties even if
they do not satisfy the criterion. Moreover, in Table 4, we give examples of families of completely decomposable Jacobian varieties arising from intermediate coverings, therefore the criterion of \cite{paola} cannot be directly applied. It is a work in progress to adjust the criterion to this situation.

We show with one example how the families on Table 3 may correspond to special subvarieties. Let $G$ be the alternating group $A_4$, acting on genus $4$ with signature $m=[0;2,3,3,3]$. We have then a one dimensional family $\mathcal{J}$ of Jacobians. Using \cite{brr} we determine that the dimension of $\mathbb{H}_4^G$ is also $1$. Therefore, according to \cite{paola}, the closure $Z$ of $\mathcal{J}$ is a special subvariety of $\A_4$ contained in $\mathcal{T}_4$ and such that $Z\cap \mathcal{T}_4^0\neq \emptyset$.

Using the group algebra decomposition (see the Appendix), we conclude that the elements in $\mathcal{J}$ (hence in $\mathbb{H}_4^G$) decompose as $E\times E_1^3$. Therefore it is an example answering \cite[Question 6.6]{mo}.  Notice that $E$ corresponds to an irreducible representation $\varphi$ of $G$ such that $G/\ker(\varphi)\cong \mathbb{Z}/3\mathbb{Z}$, hence $E$ has the action of the cyclic group of order $3$ and thus it is fixed along the family.  This family is one of the special subvarieties found in \cite[Table 2]{paola}.

\tiny
\begin{center}
\noindent \begin{longtable}{c c  l l  }
\caption{Examples of families of completely decomposable curves found through group
  algebra method. }\label{Tb:familypurple} \\
 $g$ & { Automorphism Group} & {  Signature} & { 
  Jacobian Decomposition}\\  \hline
  \endfirsthead
\caption{(continued)} \\ 
 $g$ & {  Automorphism Group} & {  Signature} & {  
  Jacobian Decomposition}\\   \hline
\endhead
11 & $(24, 14)$ & $[0;2,2,2,2,6]$ & $1, 1, 1, 2, 2, 2, 2$\\ 
 & $(48, 38)$ & $[0;2,2,2,12]$ &  1, 2, 2, 2, 4\\  
 13  & $(144, 183)$ &   $[0;2,2,2,3]$ & 2, 2, 3, 6 \\
& (48, 51) &   $[0;2,2,2,2,2]$ & 1, 1, 1, 1, 1, 2, 2, 2, 2\\ 
15   & $(48, 48)$ & $[0;2,2,4,6]$ & 1, 2, 3, 3, 3, 3 \\ 
16  & $(36, 13)$ & $[0;2,2,2,2,6]$ &  $1, 1, \underbrace{2, \ldots, 2}_{7}$ \\ 
 17  & $(192, 956)$  &  $[0;2,2,2,3]$ & $ 2, 3, 6, 6$ \\ 
& (64, 211) &  $[0;2,2,2,2,2]$ & $1, 1, 1, 1, 1, \underbrace{2, \ldots, 2}_{6}$ \\
19 & (144, 109) &  $[0;2,2,2,4]$ & 1, 3, 3, 6, 6 \\
& (72, 49) & $[0;2,2,2,2,2]$ & $1, 1, 1, \underbrace{2, \ldots, 2}_{8}$ \\
25 & (288,847) & $[0;2,2,2,3]$ &   $2, 2, 3, 4, 6, 8$    \\
 28 & (324,124) & $[0;2,2,2,3]$ &  $2, 2, 2, 4, 6, 6, 6$ \\
    31 & $(144, 154)$   & $[0;2,2,2,12]$ &  $ 1, 2, 2, 2, \underbrace{4, \ldots, 4}_6$\\
33 &  (384,18136) & $[0;2,2,2,3]$   & $3, 3, 3, 8, 8, 8$\\
 37 & (432,748) & $[0;2,2,2,3] $ &  $2, 2, 2, 3, 4, 6, 6, 12$\\
49 & $(576,  8653)$ & $[0;2,2,2,3]$ & $2,2,3, 3, 6, 6, 9, 9, 9    $ \\ 
55 & $(432,537)$ & $[0;2,2,2,4]$ & $1, 3, 3,  \underbrace{6, \ldots, 6}_{8} $ \\ 
61& $(288, 629)$ & $[0;2,2,2,12]$ & $ 1, \underbrace{2, \ldots, 2}_{6},
\underbrace{4, \ldots, 4}_{12}$ \\
65 & $(768, 1090018)$ & $[0;2,2,2,3]$ & $2,3, 3, 3, 6, \underbrace{8, \ldots, 8}_{6}$ \\ 
73& $(576, 4322)$ & $[0;2,2,2,4]$ & $ 1, 2,2,\underbrace{4, \ldots, 4}_{9}, 8, 8,
8, 8  $ \\ 
82& $(972, 474)$ & $[0;2,2,2,3]$ & $2, 2, 2,4,\underbrace{6, \ldots, 6}_{6}, {12}, {12}, {12}  $ \\ 
91& $(432,686)$ & $[0;2,2,2,12]$ &  $1, 2, 2, 2,\underbrace{4, \ldots, 4}_{21}$ \\ 
97& $(1152,157665)$ & $[0;2,2,2,3]$ & $2, 2, 3, 3, 3, 6, 6, 6, 6, 8, 8, 8, 8, {12}, {16}   $ \\ 
109 & $(1296, 2940)$ & $[0;2,2,2,3]$ & $2, 2, 2, 3, 4, \underbrace{6, \ldots, 6}_{8}, {12}, {12}, {12}, {12}  $ \\  
129&1536& $[0;2,2,2,3]$ & $2, 3, 3, 3, 6, 6, 6, 6, 6, 8, 8, \underbrace{{12}, \ldots, {12}}_6$ \\ 
145&  $(1728,46119)$ & $[0;2,2,2,3]$ & $2, 2, 3, 3, 3, \underbrace{6, \ldots, 6}_{8}, \underbrace{12,\ldots, 12}_{7} $ \\ 
193&2304& $[0;2,2,2,3]$ &  $2, 2, 3, 3, 3, 6, 6, 6, 6, \underbrace{8, \ldots, 8}_{12}, {12}, {16}, {16}, {16}$  \\
\end{longtable}
\end{center}

\begin{center}
\noindent \begin{longtable}{c | c c l | l l } 
\caption{Family of completely decomposable curves found through intermediate
cover method.}\label{Tb:familyredgreen} \\
\renewcommand{\arraystretch}{1.4}
& & Automorphism &  & Subgroup & Jacobian\\
$g$ &  Large $g$  &Group &  Signature & Order, No. & Decomposition \\ \hline
  \endfirsthead
\caption{(continued)} \\ 
& & Automorphism &  & Subgroup & Jacobian\\
$g$ &  Large $g$  &Group &  Signature & Order, No. & Decomposition \\ \hline
\endhead
12&49& $(288,627)$ & $[0;2,2,2,6]$ & 29, 4 & 1, 1, 1, 1, 2, 2, 2, 2 \\  
14  &  145 & $(1728, 46119)$ & $[0;2,2,2,3]$ & 168, 8 & $\underbrace{1, \ldots, 1}_{11}, 3$ \\
18  &  145 & $(1728, 46119)$ & $[0;2,2,2,3]$ & 152, 8 & $\underbrace{1, \ldots, 1}_{11}, 2, 2, 3$\\
21  &  129 & 1536 & $[0;2,2,2,3]$ & 128, 6 & $\underbrace{1, \ldots, 1}_{7}, \underbrace{2, \ldots, 2}_7$ \\
22  &  193 & 2304 & $[0;2,2,2,3]$ & 132, 8 & $\underbrace{1, \ldots, 1}_{7}, \underbrace{2, \ldots, 2}_6, 3$\\
23&49& $(256,3066)$ &$[0;2,2,2,8]$ & 9, 2& $1, 1, 1, \underbrace{2, \ldots, 2}_{10}$  \\ 
24  &  49 & $(288, 627)$ & $[0;2,2,2,6]$ & 8, 2 & $\underbrace{1, \ldots, 1}_{8}, \underbrace{2, \ldots, 2}_8$  \\
26  &  109 & $(1296, 2940)$ & $[0;2,2,2,3]$ & 22, 4 & $\underbrace{1, \ldots, 1}_7, 2, 2, 2, 2, 2, 3, 3, 3$ \\
27 & 55 & $(432,537)$ & $[0;2,2,2,4]$ & 6, 2 & $1, 2,  \underbrace{3, \ldots, 3}_8$ \\ 
29  &  97 & $(1152, 157665)$ & $[0;2,2,2,3]$ & 13, 3 & $1, 1, 1, \underbrace{2, \ldots, 2}_8, 4, 6$\\
33  &  193 & 2304 & $[0;2,2,2,3]$ & 74, 6 & $ \underbrace{1, \ldots, 1}_{7}, \underbrace{2, \ldots, 2}_{13}$\\
34&73& $(432,682)$& $[0;2,2,2,6]$ &5, 2& $1, 1, \underbrace{2,\ldots, 2}_{16}$ \\ 
35  &  145 & $(1728, 46119)$ & $[0;2,2,2,3]$ & 54, 4 & $\underbrace{1, \ldots, 1}_{10}, 2, 2, 2, 3,  4, 4, 4, 4$ \\
40  &  82 & $(972, 474)$ & $[0;2,2,2,3]$ & 4, 2 & $1, 1, 2, \underbrace{3, \ldots, 3}_6, 6, 6, 6$\\
41  &  129 & 1536 & $[0;2,2,2,3]$ & 15, 3 & $1, 1, 1, \underbrace{2, \ldots,  2}_7,  \underbrace{4, \ldots, 4}_6$\\
43  &  193 & 2304 & $[0;2,2,2,3]$ & 34, 4 & 1, 1, 1, 1, 2, 2, 2, 3, 4, 4, 4, 4, 6, 8\\
44&91& $(432,686)$ & $[0;2,2,2,12]$ &7, 2&  $1,  1,  \underbrace{2,\cdots, 2}_{ 21 } $ \\ 
45&91& $(432,686)$ & $[0;2,2,2,12]$ &8, 2 & $1,  1, 1,  \underbrace{2,\cdots, 2}_{ 21 } $ \\ 
46  &  109 & $(1296, 2940)$ & $[0;2,2,2,3]$ & 3, 2 & $1, 1, 1, \underbrace{2,\cdots, 2}_{8},  3, 6, 6, 6, 6$\\
47&97& $(1152,157665)$  & $[0;2,2,2,3]$ & 10, 2 & $1, 1, 1, 1, 1, 3, 3, 3, 3,
4, 4, 4, 4, 6, 8$ \\ 
52  &  109 & $(1296, 2940)$ & $[0;2,2,2,3]$ & 5, 2 & $1, 1, 1, 2, 2, \underbrace{3, \ldots, 3}_7, 6, 6, 6, 6$\\
53  &  109 & $(1296, 2940)$ & $[0;2,2,2,3]$ & 6, 2 & $1, 1, 1, 2, \underbrace{3, \ldots, 3}_{8}, 6, 6, 6, 6$\\
57  &  193 & 2304 & $[0;2,2,2,3]$ & 15, 3 & $1, 1, 1, \underbrace{2, \ldots, 2}_{16} 4, 6, 6, 6$\\
67  &  145 & $(1728, 46119)$ & $[0;2,2,2,3]$ & 11, 2 & $1, 1, 1, 1, 1, 2,  2, 2,  2, 3,  3,  3, 3, \underbrace{6, \ldots, 6}_7$\\
69  &  145 & $(1728, 46119)$ & $[0;2,2,2,3]$ & 12, 2 & $1, 1, 1, 1,  2, \underbrace{3, \ldots, 3}_{7},  \underbrace{6, \ldots, 6}_7$\\
93&193&2304 & $[0;2,2,2,3]$ & 11, 2& $1, 1, 1, 1, 2, 3, 3, 3, \underbrace{4, \ldots, 4}_{12}, 6, 8, 8, 8$ \\ 
95&193 &2304 & $[0;2,2,2,3]$ &12, 2 & $1, 1, 1, 1, 1, 3, 3, 3, 3,  \underbrace{4,\ldots, 4}_{12},  6, 8, 8, 8$ \\ 
\end{longtable}
\end{center}
\normalsize

\normalsize

\section{Complications}\label{S:complications}

The techniques described above do not necessarily guarantee the finest  decomposition.  In \eqref{eq2.4}, it is possible that $B_i \sim B_j$ even if $i \neq j$, or that the $B_i$ may decompose further.  There may be examples using a finer decomposition which fill other gaps in  Ekedahl and Serre's list. \\

Computationally, finding automorphism groups and signatures in high genus
is resource heavy. The memory requirements for the
Magma command \verb+LowIndexNormalSubgroups+
limit  our ability to use this command to find other examples in even
higher genus, or to fill remaining gaps using the intermediate cover technique. We are optimistic that, given sufficient computational resources, the
techniques we describe above could produce numerous additional new
examples.  \\

\section{Appendix}

Here we provide all examples of families of curves, for genus 3--10, which have completely decomposable Jacobians.  These were found by searching all Breuer's data for these genera, and then removing those groups that were not the full automorphism group for the given family.   Some of the examples in this table were known before \cite{paola}.

\tiny
\begin{center}
\noindent \begin{longtable}{c c  l l  }
\caption{Family of completely decomposable curves found through group
  algebra method for genus 3-10. }\label{Tb:familygenus3to10} \\
 $g$ & {  Automorphism Group} & { Signature} & { 
  Jacobian Decomposition}\\  \hline
  \endfirsthead
\caption{(continued)} \\ 
 $g$ & {  Auto. Group} & { Signature} & { 
  Jacobian Decomposition}\\   \hline
\endhead
 3    & $(4, 2)$ & $[0;2,2,2,2,2,2]$ & 1, 1, 1 \\
   & $(6, 1)$ & $[0;2,2,2,2,3]$ & 1, 2 \\
  &  $(8, 2)$ & $[0;2,2,4,4]$ & 1, 1, 1 \\
  & $(8, 5)$ & $[0;2,2,2,2,2]$ & 1, 1, 1 \\
  &  $(12,4)$ & $[0;2,2,2,6]$ & 1, 2  \\
  &  $(16, 11)$ & $[0;2,2,2,4]$ & 1, 2  \\
  &  $(16, 13)$ & $[0;2,2,2,4]$ & 1, 2 \\
  &  $(24, 12)$ & $[0;2,2,2,3]$ & 3 \\ \hline
4 &   $(8, 3)$ & $[0;2,2,2,2,4]$ & 1, 1, 2 \\ 
&  $(12,3)$ & $[0,2,3,3,3]$ & 1, 3  \\ 
&  $(12,4)$ & $[0;2,2,3,6]$ & 2, 2 \\ 
&  $(12,4)$ & $[0;2,2,2,2,2]$ & 1, 1, 2 \\ 
&  $(24, 12)$ & $[0;2,2,2,4]$ &1, 3 \\ 
&  $(36, 10)$ & $[0;2,2,2,3]$ & 2, 2 \\  \hline
  5 &  $(8, 5) $&   $[0;2,2,2,2,2,2]$ &  1, 1, 1, 1, 1  \\
&  $(12,4)$& $[0;2,2,2,2,3]$ &1, 2, 2 \\
&  $(16, 3)$& $[0;2,2,4,4]$ &1, 2, 2 \\
&  $(16, 3)$& $[0;2,2,4,4]$  & 1, 1, 1, 2 \\
&  $(16, 11)$& $[0;2,2,2,2,2]$ & 1, 1, 1, 2 \\
&  $(16, 11)$& $[0;2,2,2,2,2]$ &1, 2, 2 \\
&  $(16, 14)$& $[0;2,2,2,2,2]$ & 1, 1, 1, 1, 1 \\
&  $(24, 12)$& $[0;2,2,3,3]$ & 2, 3 \\
&  $(24, 8)$& $[0;2,2,2,6]$ & 1, 2, 2 \\
&  $(24, 14)$& $[0;2,2,2,6]$ &1, 2, 2 \\
&  $(32,27)$& $[0;2,2,2,4]$ &1, 2, 2 \\
&  $(32,28)$& $[0;2,2,2,4]$ & 1, 2, 2 \\
&  $(32,43)$& $[0;2,2,2,4]$ &  1, 4 \\ \hline
6  &  $(12, 4)$ & $[0;2,2,2,2,6]$ & 1, 1, 2, 2 \\
& $(24, 12)$ & $[0;2,2,3,4]$ & 3, 3 \\  \hline
7  & $(8, 5)$ & $[0;2,2,2, 2,2,2,2]$ & 1, 1, 1, 1, 1, 1, 1\\
& $(16, 11)$ & $[0;2,2,2,2,4]$ & 1, 1, 1, 2, 2 \\
 & $(18, 4)$ & $[0;2,2,2,2,3]$ &  1, 2, 2, 2 \\
&  $(24, 13)$ & $[0;2,2,3,6]$ & 1, 3, 3 \\
&  $(24, 14)$ & $[0;2,2,2,2,2]$ & 1, 1, 1, 2, 2 \\
&  $(32, 43)$ & $[0;2,2,2,8]$ & 1, 2, 4 \\
&  $(36, 10)$ & $[0;2,2,2,6]$ & 1, 2, 4 \\
&  $(48, 38)$ & $[0;2,2,2,4]$ & 1, 2, 4 \\
&  $(48, 48)$ & $[0;2,2,2,4]$ & 1, 3, 3 \\ \hline
8 &  $(24, 12)$ & $[0;2,3,3,4]$ & 2, 3, 3 \\ \hline
9 &   $(16, 11)$ & $[0;2,2,2,2,2,2]$ & 1, 1, 1, 1, 1, 2, 2 \\
&  $(16, 14)$ & $[0;2,2,2,2,2,2]$ & 1, 1, 1, 1, 1, 1, 1, 1, 1 \\
&  $(24, 14)$ & $[0;2,2,2,2,3]$ & 1, 2, 2, 2, 2 \\
&  $(32, 6)$ & $[0;2,2,4,4]$ & 1, 2, 2, 4 \\
&  $(32, 27)$ & $[0;2,2,2,2,2]$ &1, 1, 1, 2, 2, 2 \\
&  $(32, 34)$ & $[0;2,2,2,2,2]$ & 1, 2, 2, 2, 2 \\
&  $(32, 43)$ & $[0;2,2,2,2,2]$ &  1, 1, 1, 2, 2 \\
&  $(32, 46)$ & $[0;2,2,2,2,2]$ &  1, 1, 1, 1, 1, 2, 2 \\
&  $(32, 49)$ & $[0;2,2,2,2,2]$ & 1, 1, 1, 1, 1, 4 \\
&  $(48, 38)$ & $[0;2,2,2,6]$ & 1, 2, 2, 4\\
&  $(48, 43)$ & $[0;2,2,2,6]$ & 1, 2, 2, 2, 2  \\
&  $(48, 48)$ & $[0;2,2,2,6]$ &  3, 3, 3  \\
&  $(64, 73)$ & $[0;2,2,2,4]$ &  1, 2, 2, 2, 2 \\
&  $(64, 128)$ & $[0;2,2,2,4]$ &  1, 2, 2, 4 \\
&  $(64, 134)$ & $[0;2,2,2,4]$ &  1, 2, 2, 4\\
&  $(64, 135)$ & $[0;2,2,2,4]$ &  1, 2, 2, 4 \\
&  $(64, 138)$ & $[0;2,2,2,4]$ & 1, 2, 2, 4 \\
&  $(64, 140)$ & $[0;2,2,2,4]$ & 1, 2, 2, 4 \\
&  $(64, 177)$ & $[0;2,2,2,4]$ & 1, 4, 4 \\
&  $(96, 193)$ & $[0;2,2,2,3]$ & 2, 3, 4 \\
&  $(96, 227)$ & $[0;2,2,2,3]$ & 3, 3, 3 \\ \hline
10 &  $(36, 10)$ & $[0;2,2,3,6]$ & 2, 2, 2, 4 \\
&  $(36, 13)$ & $[0;2,2,3,6]$ &  2, 2, 2, 2, 2 \\
&  $(36, 10)$ & $[0;2,2,2,2,2]$ &  1, 1, 2, 2, 4 \\
&  $(36, 13)$ & $[0;2,2,2,2,2]$ & 1, 1, 2, 2, 2, 2  \\
&  $(48, 29)$ & $[0;2,2,2,8]$ & 1, 2, 3, 4\\
&  $(72, 15)$ & $[0;2,2,2,4]$ &  1, 3, 6  \\
&  $(72, 40)$ & $[0;2,2,2,4]$ &  2, 4, 4\\
&  $(72, 43)$ & $[0;2,2,2,4]$ &  1, 3, 6 \\
&  $(108, 17)$ & $[0;2,2,2,3]$ &  2, 2, 6 \\
&  $(108, 40)$ & $[0;2,2,2,3]$ & 2, 2, 2, 4 \\ \hline
\end{longtable}
\end{center}

\end{document}